\documentclass[a4paper,11pt,reqno]{article}

\usepackage[utf8]{inputenc}
\usepackage[T1]{fontenc}
\usepackage{lmodern}
\usepackage[english]{babel}
\usepackage{microtype}

\usepackage{amsmath,amssymb,amsfonts,amsthm}
\usepackage{mathtools,accents}
\usepackage{mathrsfs}
\usepackage{aliascnt}
\usepackage{braket}
\usepackage{bm}

\usepackage[a4paper,margin=3cm]{geometry}
\usepackage[citecolor=blue,colorlinks]{hyperref}

\usepackage{enumerate}
\usepackage{xcolor}

\makeatletter
\g@addto@macro\@floatboxreset\centering
\makeatother

\makeatletter
\def\newaliasedtheorem#1[#2]#3{
  \newaliascnt{#1@alt}{#2}
  \newtheorem{#1}[#1@alt]{#3}
  \expandafter\newcommand\csname #1@altname\endcsname{#3}
}
\makeatother

\numberwithin{equation}{section}

\newtheoremstyle{slanted}{\topsep}{\topsep}{\slshape}{}{\bfseries}{.}{.5em}{}

\theoremstyle{plain}
\newtheorem{theorem}{Theorem}[section]
\newaliasedtheorem{proposition}[theorem]{Proposition}
\newaliasedtheorem{lemma}[theorem]{Lemma}
\newaliasedtheorem{corollary}[theorem]{Corollary}
\newaliasedtheorem{conjecture}[theorem]{Conjecture}
\newaliasedtheorem{counterexample}[theorem]{Counterexample}

\theoremstyle{definition}
\newaliasedtheorem{definition}[theorem]{Definition}
\newaliasedtheorem{question}[theorem]{Question}
\newaliasedtheorem{example}[theorem]{Example}

\theoremstyle{remark}
\newaliasedtheorem{remark}[theorem]{Remark}

\let\altphi\phi
\let\phi\varphi
\let\varphi\altphi
\let\altphi\undefined

\newcommand{\di}{\mathop{}\!\mathrm{d}}

\newcommand{\res}{\mathop{\hbox{\vrule height 7pt width .5pt depth 0pt
\vrule height .5pt width 6pt depth 0pt}}\nolimits}

\DeclareMathOperator{\supp}{supp}

\newcommand{\Ch}{{\sf Ch}}

\DeclareMathOperator{\Lip}{Lip}

\newcommand{\dist}{\mathsf{d}}

\newcommand{\meas}{\mathfrak{m}}

\DeclareMathOperator{\RCD}{RCD}
\DeclareMathOperator{\CD}{CD}

\newfont{\tmpf}{cmsy10 scaled 2500}

\begin{document}
\title{New differential operator and non-collapsed $\RCD$ spaces}
\author{Shouhei Honda
\thanks{Tohoku University, \url{shonda@m.tohoku.ac.jp}, \url{shouhei.honda.e4@tohoku.ac.jp}.}} 
\maketitle
\centerline{\textit{Dedicated to Professor Kenji Fukaya on the occasion of his sixtieth birthday.}}
\begin{abstract}
We show characterizations of non-collapsed compact $\RCD(K, N)$ spaces, which in particular confirm a conjecture of De Philippis-Gigli on the implication from the weakly non-collapsed condition to the non-collapsed one in the compact case.
The key idea is to give the explicit formula of the Laplacian associated to the pull-back Riemannian metric by embedding in $L^2$ via the heat kernel. This seems the first application of geometric flow to the study of $\RCD$ spaces.
\end{abstract}

\tableofcontents

\section{Introduction}
\subsection{Main results}\label{mm}

De Philippis-Gigli introduced in \cite{DePhillippisGigli} two special classes of $\RCD(K, N)$ spaces.
One of them is the notion of \textit{weakly non-collapsed spaces} and the other one is that of \textit{non-collapsed spaces}.
Our main result states that these are essentially same in the compact case.

After the fundamental works of Lott-Villani \cite{LottVillani} and Sturm \cite{Sturm06}, Ambrosio-Gigli-Savar\'e \cite{AmbrosioGigliSavare14} (when $N=\infty$), Gigli \cite{Gigli13} and Erbar-Kuwada-Sturm \cite{ErbarKuwadaSturm} (when $N<\infty$) introduce the notion of $\RCD(K, N)$ \textit{spaces} for metric measure spaces $(X, \dist, \meas)$, which means a synthetic notion of ``$\mathrm{Ric} \ge K$ and $\mathrm{dim}\le N$ with Riemannian structure''.
Typical examples are measured Gromov-Hausdorff limit spaces of Riemannian manifolds with Ricci bounds from below and dimension bounds from above, so-called \textit{Ricci limit spaces}. The $\RCD$ theory gives the striking framework to treat Ricci limit spaces by a synthetic way.

Cheeger-Colding established the fundamental structure theory of Ricci limit spaces \cite{CheegerColding1}.
Thanks to recent quick developments on the study of $\RCD(K, N)$ spaces, the most part of the theory of Ricci limit spaces, including Colding-Naber's result \cite{CN}, are covered by the $\RCD$ theory (see for instance \cite{BrueSemola} by Bru\`e-Semola). In particular whenever $N<\infty$, the essential dimension, denoted by $\mathrm{dim}_{\dist, \meas}(X)$, of any $\RCD(K, N)$ space $(X, \dist, \meas)$ makes sense (c.f. Theorem \ref{th: RCD decomposition}). 

On the other hand in a special class of Ricci limit spaces, so-called \textit{non-collapsed Ricci limit spaces},  finer properties are obtained by Cheeger-Colding. For instance, the Bishop inequality with the rigidity and the almost Reifenberg flatness are justified in this setting. They are not covered by general Ricci limits/$\RCD$ theories.

The properties of non-collapsed $\RCD(K, N)$ spaces introduced in \cite{DePhillippisGigli} cover finer results on non-collapsed Ricci limit spaces as explained above. It is worth pointing out that any convex body is not a non-collapsed Ricci limit space, but it is a non-collapsed $\RCD(K, N)$ space.

Let us give the definitions; 
\begin{itemize}
\item an $\RCD(K, N)$ space $(X, \dist, \meas)$ is \textit{non-collapsed} if $\meas =\mathcal{H}^N$, where $\mathcal{H}^N$ denotes the $N$-dimensional Hausdorff measure;
\item an $\RCD(K, N)$ space $(X, \dist, \meas)$ is \textit{weakly non-collapsed} if $\meas \ll \mathcal{H}^N$.
\end{itemize}
The second definition is equivalent to that $\dim_{\dist, \meas}(X)=N$, which is proved in \cite{DePhillippisGigli}.
Note that some structure results on weakly non-collapsed $\RCD(K, N)$ spaces are obtained in \cite{DePhillippisGigli} and that Kitabeppu \cite{Kit1} provides a similar notion (which is a priori stronger than the weakly non-collapsed condition, but is a priori weaker than the non-collapsed one) and prove similar structure results.

De Philippis-Gigli conjectured that these notions are essentially same. More precisely; 
\begin{conjecture}\label{con:1}
If $(X, \dist, \meas)$ is a weakly non-collapsed $\RCD(K, N)$ space, then $\meas = a\mathcal{H}^N$ for some $a \in (0, \infty)$.
\end{conjecture}

For the conjecture the only known development is due to Kapovitch-Ketterer \cite{KK} and Han \cite{Han19}. Kapovitch-Ketterer proved that Conjecture \ref{con:1} is true under assuming bounded sectional curvature from above in the sense of Alexandrov (that is, the metric structure is CAT). Han proved that this conjecture is true for smooth Riemannian manifolds with (not necessary smooth) weighted measures.

We are now in a position to introduce a main result of the paper.
\begin{theorem}[Characterization of non-collapsed $\RCD$ spaces]\label{thm:equiv}
Let $(X, \dist, \meas)$ be a compact $\RCD(K, N)$ space with $n:=\mathrm{dim}_{\dist, \meas}(X)$.
Then the following two conditions (1), (2) are equivalent.
\begin{enumerate}
\item The following two conditions hold;
\begin{enumerate}
\item For all eigenfunction $f$ on $X$ of $-\Delta$ we have
\begin{equation}\label{eq:lap}
\Delta f= \mathrm{tr}(\mathrm{Hess}_f) \qquad \text{in $L^2(X, \meas)$.}
\end{equation}
\item There exists $C>0$ such that 
\begin{equation}\label{eq:vol}
\meas (B_r(x)) \ge Cr^n \quad \forall x\in X,\,\,\forall r \in (0,1).
\end{equation}
\end{enumerate}
\item $(X, \dist, \meas)$ is an $\RCD(K, n)$ space with 
\begin{equation}\label{eq:meas}
\meas = \frac{\meas (X)}{\mathcal{H}^n(X)} \mathcal{H}^n.
\end{equation}
\end{enumerate}
\end{theorem}

It is easy to understand that this theorem gives a contribution to Conjecture \ref{con:1}.
More precisely combining a result of Han \cite{Han} (c.f. Theorem \ref{thm:hanresult}) with the Bishop-Gromov inequality yields that all compact weakly non-collapsed $\RCD(K, N)$ spaces satisfy (1) in the theorem as $n=N$.
Therefore
\begin{corollary}\label{cor:main}
Conjecture \ref{con:1} is true in the compact case.
\end{corollary}

We will also establish other characterization of non-collapsed $\RCD$ spaces. See subsection \ref{4.2}.  Next let us explain how to achieve these results. Roughly speaking it is to take \textit{canonical deformations $g_t$} of the Riemannian metric $g$ via the heat kernel.
\subsection{Key idea; deformation of Riemannian metric via heat kernel}
In order to prove main results the key idea is to use the pull-back Riemannian metrics $g_t:=\Phi_t^*g_{L^2}$ by embeddings $\Phi_t: X \to L^2(X, \meas)$ via the heat kernel $p$ instead of using the original Riemannian metric $g$ of $(X, \dist, \meas)$. The definition of $\Phi_t$ is
\begin{equation}
\Phi_t(x) (y):=p(x, y, t).
\end{equation}
This map is introduced and studied by B\'erard-Besson-Gallot \cite{BerardBessonGallot} for closed manifolds. They proved that for closed manifolds $(M^n, g)$, as $t \to 0^+$
\begin{equation}
\omega_nt^{(n+2)/2}g_t=c_ng-\frac{c_n}{3}\left(\mathrm{Ric}_g-\frac{1}{2}\mathrm{Scal}_gg\right)t +O(t^2),
\end{equation}
where $\mathrm{Ric}_g$ and $\mathrm{Scal}_g$ denote the Ricci and the scalar curvatures respectively, and 
\begin{equation}\label{eq:const}
\omega_n:=\mathcal{L}^n(B_1(0_n)), \quad c_n:=\frac{\omega_n}{(4\pi)^n}\int_{\mathbb{R}^n}|\partial_{x_1}(e^{-|x|^2/4})|^2\di \mathcal{L}^n(x).
\end{equation}
Recently the map $\Phi_t$ is also studied for compact $\RCD(K, N)$ spaces by Ambrosio-Portegies-Tewodrose and the author \cite{AHPT}. In particular $g_t$ is also well-defined in this setting (c.f. Theorem \ref{prop:riemanexist}).

Let us introduce the following new differential operator
\begin{equation}
\Delta^tf:=\langle \mathrm{Hess}_f, g_t\rangle +\frac{1}{4}\langle \nabla_x\Delta_xp(x, x, 2t), \nabla f\rangle.
\end{equation}
This plays a role of the Laplacian associated to $g_t$, in fact, we will prove
\begin{equation}\label{eq:010}
\int_X\langle g_t, \dist \psi\otimes \dist f\rangle\di \meas =-\int_X\psi\Delta^tf\di \meas,
\end{equation}
which is new even for closed manifolds.
See Theorem \ref{thm:green} for the precise statement.
Then under assuming (\ref{eq:vol}), after nomalization, taking the limit $t\to 0^+$ in (\ref{eq:010}) with convergence results given in \cite{AHPT} yields the \textit{metric} integration by parts formula
\begin{equation}\label{eq:b}
\int_X\langle \nabla \psi, \nabla f\rangle\di \mathcal{H}^n =-\int_X\psi \mathrm{tr}(\mathrm{Hess}_f)\di \mathcal{H}^n.
\end{equation}
This allows us to prove (\ref{eq:meas})  by letting $\psi \equiv 1$ if in addition (\ref{eq:lap}) holds because we see that $\frac{\di \mathcal{H}^n}{\di \meas}$ is $L^2(X, \meas)$-orthogonal to $\mathrm{tr} (\mathrm{Hess}_f)$ for all $f$, in particular, $\frac{\di \mathcal{H}^n}{\di \meas}$ is $L^2(X, \meas)$-orthogonal to any nontrivial eigenfunction. This implies that $\frac{\di \mathcal{H}^n}{\di \meas}$ must be a constant function.

Finally let us give few comments. It is well-known that in the smooth setting, there are many geometric flows (e.g. Ricci flow) which are useful to understand the original space.
However for singular spaces there are not so many (e.g. \cite{BK}, \cite{GM}, \cite{KL1} and \cite{KL2}). In general $\RCD$ spaces have very wild singularities (e.g. the singular set may be dense). This paper shows us that such flow approaches are also useful even in the $\RCD$ setting. Geometric applications of the main results can be found in \cite{HM, KM}.
Moreover although we discuss only in the compact case, the author believes that the techniques provided
will be available even in the noncompact case. 

The paper is organized as follows: 

In Section \ref{se:rcdstarkn} we give a quick introduction on $\RCD$ spaces and prove technical results. 
In Section \ref{sec3} we establish (\ref{eq:010}).
In the final section, Section \ref{sec4}, we prove main results stated in subsection \ref{mm} and related results.
It is worth pointing out that subsection \ref{4.2} is written from the point of view of metric geometry.

\smallskip\noindent
\textbf{Acknowledgement.}
The author is grateful to the referees for careful readings and valuable suggestions.
He acknowledges supports of the Grantin-Aid
for Young Scientists (B) 16K17585 and Grant-in-Aid for Scientific Research (B) of 18H01118.

\section{$\RCD (K,N)$ spaces}\label{se:rcdstarkn}
A triple $(X, \dist, \meas)$ is a \textit{metric measure space} if $(X, \dist)$ is a complete separable metric space and $\meas$ is a Borel measure on $X$ with $\supp \meas =X$. 
\subsection{Definition}\label{2.1}
Throughout this paper the parameters $K\in\mathbb{R}$ (lower bound on Ricci curvature) and $N \in [1, \infty)$ (upper bound on dimension) will be kept fixed. 
Instead of giving the original definition of $\RCD(K, N)$ spaces, we introduce an equivalent version for short.
See \cite{ErbarKuwadaSturm}, \cite{AmbrosioMondinoSavare}, \cite{CM} and \cite{AmbrosioGigliSavare13} for the proof of the equivalence and the detail.

Let $(X, \dist, \meas)$ be a metric measure space.
The Cheeger energy
$\Ch:L^2(X,\meas)\to [0,+\infty]$ is a convex and $L^2(X,\meas)$-lower semicontinuous functional defined as follows:
\begin{equation}\label{eq:defchp}
\Ch(f):=\inf\left\{\liminf_{n\to\infty}\frac 12\int_X(\mathrm{Lip}  f_n)^2\di\meas:\ \text{$f_n\in\Lip_b (X,\dist) \cap L^2(X, \meas)$, $\|f_n-f\|_{L^2}\to 0$}\right\}, 
\end{equation}
where $\Lip_b (X,\dist)$ denotes the space of all bounded Lipschitz functions and $\mathrm{Lip} f$ is the slope.

The Sobolev space $H^{1,2}(X,\dist,\meas)$ then coincides with $\{f\in L^2(X, \meas): \Ch(f)<+\infty\}$. When endowed with the norm
$\|f\|_{H^{1,2}}:=(\|f\|_{L^2(X,\meas)}^2+2\Ch(f))^{1/2}$,
this space is Banach, reflexive if $(X,\dist)$ is doubling (see \cite[Cor.7.5]{AmbrosioColomboDiMarino}), and  
separable Hilbert if $\Ch$ is a quadratic form (see \cite[Prop.4.10]{AmbrosioGigliSavare14}). 
According to the terminology introduced in \cite{Gigli1}, we say that $(X,\dist,\meas)$ is infinitesimally Hilbertian if $\Ch$ is a quadratic form.

Let us assume that $(X, \dist, \meas)$ is infinitesimally Hilbertian. Then
for all $f_i \in H^{1, 2}(X, \dist, \meas),$
\begin{equation}
\langle \nabla f_1, \nabla f_2\rangle:=\lim_{\epsilon \to 0}\frac{|\nabla (f_1+\epsilon f_2)|^2-|\nabla f_1|^2}{2\epsilon} \in L^1(X, \meas)
\end{equation}
is well-defined,
where $|\nabla f| \in L^2(X, \meas)$ denotes the minimal relaxed slope of $f \in H^{1, 2}(X, \dist, \meas)$ (see \cite[Sect. $3$ and $4$]{Gigli1}).
  

We can now define a densely
defined operator $\Delta:D(\Delta)\to L^2(X,\meas)$ whose domain consists of all functions $f\in H^{1,2}(X,\dist,\meas)$
satisfying
$$
 \int_X \psi \phi \di \meas=-\int_X \langle \nabla f, \nabla \phi \rangle \di \meas\quad\qquad\forall \phi \in H^{1,2}(X,\dist,\meas)
$$
for some $\psi \in L^2(X,\meas)$. The unique $\psi$ with this property is then denoted by $\Delta f$.

We are now in a position to introduce the $\RCD$ space:
\begin{definition}[$\RCD$ spaces]\label{bedef}
Let $(X, \dist, \meas)$ be a metric measure space, let $K \in \mathbb{R}$ and let $\hat{N} \in [1, \infty]$.
We say that $(X, \dist, \meas)$ is an \textit{$\RCD (K, \hat{N})$ space} if the following hold;
\begin{enumerate}
\item{(Volume growth)} there exist $x \in X$ and $C>1$ such that $\meas (B_r(x)) \le Ce^{Cr^2}$ for all $r \in (0, \infty)$;
\item{(Bochner's inequality)}
for all $f\in D(\Delta)$ with $\Delta f\in H^{1,2}(X,\dist,\meas)$, 
\begin{equation}\label{eq:boch}
\frac 12\int_X |\nabla f|^2\Delta\phi\dist\meas\geq
\int_X\phi\left(\frac{(\Delta f)^2}{\hat{N}}+ \langle \nabla f,\nabla \Delta f\rangle + K|\nabla f|^2\right)\dist\meas 
\end{equation}
for all $\phi\in D(\Delta) \cap L^{\infty}(X, \meas)$ with $\phi\geq 0$ and $\Delta\phi\in L^\infty(X,\meas)$;
\item{(Sobolev-to-Lipschitz property)} any  
$f\in H^{1,2}(X,\dist,\meas)$ with $|\nabla f| \leq 1$ $\meas$-a.e. in $X$ 
has a $1$-Lipschitz representative.
\end{enumerate}
\end{definition}
It is known that for a smooth weighted complete Riemannian manifold $(M^n, g, e^{-\phi}\mathrm{vol}_g) (\phi \in C^{\infty}(M^n))$ it is an $\RCD(K, \hat{N})$ space for $K \in \mathbb{R}$ and $\hat{N} \in (n, \infty]$ if and only if it holds that
\begin{equation}\label{chc}
\mathrm{Ric}_g+\mathrm{Hess}_{\phi}^g-\frac{\dist \phi \otimes \dist \phi}{\hat{N}-n}\ge Kg.
\end{equation}
See \cite[Prop.4.21]{ErbarKuwadaSturm}. In particular if it is an $\RCD(K, n)$ space, then $\phi$ must be a constant function because it is also an $\RCD(K, \hat{N})$ space for all $\hat{N} \in (n, \infty]$, which implies by (\ref{chc}) that $|\dist \phi| \equiv 0$.
Let us denote the heat flow associated to the Cheeger energy by $h_t$. It holds (without curvature assumption) that 
\begin{equation}\label{eq:22}
\|h_tf\|_{L^2}\le \|f\|_{L^2},\qquad \||\nabla f|\|_{L^2}^2\le \frac{\|f\|_{L^2}^2}{2t^2},\qquad \|\Delta h_tf\|_{L^2} \le \frac{\| f\|_{L^2}}{t}.
\end{equation} 
Then one of the crucial properties of the heat flow on $\RCD (K, \infty)$ spaces is
\begin{equation}
h_tf \in \mathrm{Test}F(X, \dist, \meas), \qquad \forall f \in L^2(X, \meas) \cap L^{\infty}(X, \meas),\,\,\forall t \in (0, \infty),
\end{equation}
where
\begin{equation}\label{eq:deftest}
\mathrm{Test}F(X,\dist,\meas):=\left\{f\in\Lip_b(X,\dist)\cap H^{1,2}(X,\dist,\meas):\ \Delta f\in H^{1,2}(X,\dist,\meas)
\right\}.
\end{equation}
See for instance \cite{Gigli} for the crucial role of test functions in the study of $\RCD$ spaces.
Finally we end this subsection by giving the following elementary lemma.
\begin{lemma}\label{lem:density}
Let $(X, \dist, \meas)$ be an $\RCD(K, \infty)$ space and let $f \in D(\Delta)$.
Then there exists a sequence $f_i \in \mathrm{Test}F(X, \dist, \meas)$ such that $\|f_i-f\|_{H^{1, 2}} +\|\Delta f_i - \Delta f\|_{L^2} \to 0$ holds.
\end{lemma}
\begin{proof}
Let $F_L:=(-L) \vee f \wedge L$. Note that $h_tF_L \in \mathrm{Test}F(X, \dist, \meas)$, that $h_tF_L \to h_tf$ in $H^{1, 2}(X, \dist, \meas)$ as $L \to \infty$ for all $t>0$, and that $\Delta h_tF_L$ $L^2$-weakly converges to $\Delta h_t f$ as $L \to \infty$ for all $t>0$, where we used (\ref{eq:22}) (c.f. \cite[Cor.10.4]{AmbrosioHonda}). Since $\Delta h_tf \to \Delta f$ in $L^2(X, \meas)$ as $t \to 0^+$, there exist $L_i \to \infty$ and $t_i \to 0^+$ such that $h_{t_i}F_{L_i} \to f$ in $H^{1, 2}(X, \dist, \meas)$ and that $\Delta h_{t_i}F_{L_i}$ $L^2$-weakly converges to $\Delta f$.
Then applying Mazur's lemma for the sequence $\{\Delta h_{t_i}F_{L_i}\}_i$ yields that for all $m \ge 1$ there exist $N_m \in \mathbb{N}_{\ge m}$ and $\{t_{m, i}\}_{m \le i \le N_m} \subset [0, 1]$ such that $\sum_{i=m}^{N_m}t_{m, i}=1$ and that $\sum_{i=m}^{N_m}t_{m, i}\Delta h_{t_i}F_{L_i} \to \Delta f$ in $L^2(X, \meas)$. It is easy to check that $f_m:=\sum_{i=m}^{N_m}t_{m, i}h_{t_i}F_{L_i}$ satisfies the desired claim.
\end{proof}

\subsection{Heat kernel}
It is well-known that the Bishop-Gromov theorem holds for any $\RCD(K,N)$ space $(X,\dist,\meas)$ (or more generally for $\CD^*(K, N)$ spaces) and that the local Poincar\'e inequality holds for $\RCD(K, \infty)$ spaces (or more generally for $\CD(K, \infty)$ spaces). See \cite[Th.30.11]{Villani}, \cite{VR} and \cite[Th.1]{Rajala}.
Furthermore, it follows from the Sobolev-to-Lipschitz property 
that, on any $\RCD(K,N)$ space $(X,\dist,\meas)$, the intrinsic distance
$$
\dist_{\Ch}(x,y):= \sup \{ |f(x)-f(y)| \, : \, f \in H^{1,2}(X,\dist,\meas) \cap C_b(X, \dist), \, |\nabla f| \le 1 \}
$$
associated to the Cheeger energy $\Ch$ coincides with the original distance $\dist$. Consequently, applying \cite[Prop.2.3]{Sturm95} and \cite[Cor.3.3]{Sturm96} on the general theory of Dirichlet forms provide the existence of a locally H\"older continuous representative $p$ on $X \times X \times (0,\infty)$ for the heat kernel of $(X,\dist,\meas)$. Let us recall that by definition
\begin{equation}
h_tf(x)=\int_Xp(x, y, t)f(y)\di \meas(y) \quad \forall t>0,\,\forall x\in X,\,\forall f\in L^2(X, \meas)
\end{equation}
and 
\begin{equation}\label{trans}
p(x, y, t+s)=\int_Xp(x, z, t)p(z, y, s)\di \meas (z) \quad \forall t>0,\,\forall s>0,\,\forall x, y \in X.
\end{equation}
The sharp Gaussian estimates on this heat kernel have been proved later on in the $\RCD$ context \cite[Th.1.2]{JiangLiZhang}:
for any $\epsilon>0$, there exist $C_i:=C_i(\epsilon, K, N)>1$ for $i=1,\,2$, depending only on $K$, $N$ and $\epsilon$, such that 
\begin{equation}\label{eq:gaussian}
\frac{C_1^{-1}}{\meas (B_{\sqrt{t}}(x))}\exp \left(-\frac{\dist^2 (x, y)}{(4-\epsilon)t}-C_2t \right) \le p(x, y, t) \le \frac{C_1}{\meas (B_{\sqrt{t}}(x))}\exp \left( -\frac{\dist^2 (x, y)}{(4+\epsilon)t}+C_2t \right)
\end{equation}
for all $x,\, y \in X$ and any $t>0$, where from now on we state our inequalities with the H\"older continuous representative. Combining (\ref{eq:gaussian}) with the Li-Yau inequality \cite[Cor.1.5]{GarofaloMondino}, \cite[Th.1.2]{Jiang15}, 
we have a gradient estimate \cite[Cor.1.2]{JiangLiZhang}:
\begin{equation}\label{eq:equi lip}
|\nabla_x p(x, y, t)|\le \frac{C_3}{\sqrt{t}\meas (B_{\sqrt{t}}(x))}\exp \left(-\frac{\dist^2(x, y)}{(4+\epsilon) t}+C_4t\right)
\qquad\text{for $\meas$-a.e. $x\in X$}
\end{equation}
for any $t>0$, $y\in X$, where $C_i:=C_i(\epsilon, K, N)>1$ for $i=3,\,4$,
Note that in this paper, we will always work with \eqref{eq:gaussian} and \eqref{eq:equi lip} in the case $\epsilon = 1$.

Let us assume that $\mathrm{diam}(X, \dist)<\infty$, thus $(X, \dist)$ is compact (because in general $(X, \dist)$ is proper).
Then the doubling condition and a local Poincar\'e inequality on $(X, \dist, \meas)$ yields that the canonical embedding map $H^{1, 2}(X, \dist, \meas) \hookrightarrow L^2(X, \meas)$ is a compact operator \cite[Th.8.1]{HK}.
In particular the (minus) Laplacian $-\Delta$ admits a discrete positive spectrum $0=\lambda_0 < \lambda_1 \le \lambda_2 \le \cdots \to + \infty$. We denote the corresponding eigenfunctions by $\phi_0, \phi_1, \ldots $ with $\|\phi_i\|_{L^2}=1$. This provides the following expansions for the heat kernel $p$:
\begin{equation}\label{eq:expansion1}
p(x,y,t) = \sum_{i \ge 0} e^{- \lambda_i t} \phi_i(x) \phi_i (y) \qquad \text{in $C(X\times X)$}
\end{equation}
for any $t>0$ and
\begin{equation}\label{eq:expansion2}
p(\cdot,y,t) = \sum_{i \ge 0} e^{- \lambda_i t} \phi_i(y) \phi_i \qquad \text{in $H^{1,2}(X,\dist,\meas)$}
\end{equation}
for any $y\in X$ and $t>0$ with the H\"older representative of all eigenfunctions. 
Combining (\ref{eq:expansion1}) and (\ref{eq:expansion2}) with (\ref{eq:equi lip}), we know that $\phi_i$ is Lipschitz, in fact, it holds that
\begin{equation}
\|\phi_i\|_{L^\infty} \leq C_5 \lambda_i^{N/4}, \qquad \| \nabla \phi_i \|_{L^\infty} \leq C_5 \lambda_i^{(N+2)/4}, \qquad \lambda_i \ge C_5^{-1}i^{2/N},
\end{equation}
where $C_5:=C_5(\mathrm{diam} (X, \dist), K, N)>0$.
See for instance appendices in \cite{AHPT} and \cite{Honda0} for the proofs.

Finally let us remark that it follows from this observation with (\ref{eq:pp}) that
\begin{equation}\label{eq:100}
\Delta_xp(x, x, t) = 2\sum_{i \ge 0}e^{-\lambda_it} \left(-\lambda_i(\phi_i(x))^2 +|\nabla \phi_i|^2(x)\right) \qquad \text{in $H^{1, 2}(X, \dist, \meas)$}
\end{equation}
holds because (\ref{eq:pp}) implies that 
\begin{equation}\label{pppppppppp}
\sup_k \left\|\Delta \left(\sum_{i \ge 0}^ke^{-\lambda_it}\phi_i^2\right)\right\|_{H^{1, 2}}= \sup_k\left\|2\sum_{i \ge 0}^ke^{-\lambda_it} \left(-\lambda_i\phi_i^2 +|\nabla \phi_i|^2\right)\right\|_{H^{1, 2}}<\infty.
\end{equation}
In particular, thanks to (\ref{eq:expansion1}), we see that $p(x, x, t) \in D(\Delta)$ holds that the equality in (\ref{eq:100}) is satisfied in $L^2(X, \meas)$. Then applying Mazur's lemma for the sequence $\{\Delta (\sum_{i \ge 0}^ke^{-\lambda_it}\phi_i^2)\}_k$ with (\ref{pppppppppp}) allows us to prove that the equality in (\ref{eq:100}) holds in $H^{1, 2}(X, \dist, \meas)$.
\subsection{Infinitesimal structure}
Let $(X, \dist, \meas)$ be an $\RCD(K, N)$ space.
\begin{definition}[Regular set $\mathcal{R}_k$]
For any $k \geq 1$, we denote by $\mathcal{R}_k$ the \textit{$k$-dimensional regular set}  of $(X, \dist, \meas)$, 
namely the set of points $x \in X$ such that $(X, r^{-1}\dist, \meas (B_{r}(x))^{-1}\meas, x)$ pointed measured Gromov-Hausdorff converge to $(\mathbb{R}^k, \dist_{\mathbb{R}^k}, \omega_k^{-1}\mathcal{L}^k,0_k)$ as $r \to 0^+$.
\end{definition}

We are now in a position to introduce the latest structural result for $\RCD(K,N)$ spaces.

\begin{theorem}[Essential dimension of $\RCD (K,N)$ spaces]\label{th: RCD decomposition}
Let $(X,\dist,\meas)$ be an $\RCD (K,N)$ space. Then, there exists a unique integer $n\in [1,N]$, denoted by $\dim_{\dist,\meas}(X)$, such that
 \begin{equation}\label{eq:regular set is full}
\meas(X\setminus \mathcal{R}_n\bigr)=0.
\end{equation}
In addition, the set $\mathcal{R}_n$ is $(\meas,n)$-rectifiable and $\meas$ is representable as
$\theta\mathcal{H}^n\res\mathcal{R}_n$ for some nonnegative valued function $\theta \in L^1_{\mathrm{loc}}(X, \mathcal{H}^n)$.
\end{theorem}
 Note that the rectifiability of all sets $\mathcal{R}_k$ was inspired by \cite{CheegerColding1} and proved in \cite[Th.1.1]{MondinoNaber}, together with the concentration property
$\meas(X\setminus\cup_k\mathcal{R}_k)=0$, with the crucial uses of \cite{GigliMondinoRajala} and of \cite{Gigli13}; the absolute continuity of $\meas$ on regular sets with respect to the corresponding Hausdorff measure was proved afterwards and is a consequence of \cite[Th.1.2]{KellMondino}, \cite[Th.1.1]{DePhillippisMarcheseRindler} and \cite[Th.3.5]{GigliPasqualetto}. Finally, in the very recent work \cite[Th.0.1]{BrueSemola} it is proved that only one set $\mathcal{R}_n$ has positive
$\meas$-measure, leading to \eqref{eq:regular set is full} and to the representation $\meas=\theta\mathcal{H}^n\res\mathcal{R}_n$. Recall that our main target of the paper is $\theta$.

By slightly refining the definition of $n$-regular set, passing to a reduced set $\mathcal{R}_n^*$, general results of measure differentiation provide also  
the converse absolute continuity property $\mathcal{H}^n \ll \meas$ on $\mathcal{R}_n^*$. We summarize here the results obtained in this direction in
\cite[Th.4.1]{AmbrosioHondaTewodrose}:

\begin{theorem}[Weak Ahlfors regularity]\label{thm:RN}
Let $(X, \dist, \meas)$ be an $\RCD (K, N)$-space, $n=\dim_{\dist,\meas}(X)$, $\meas=\theta\mathcal{H}^n\res\mathcal{R}_n$ and set
\begin{equation}\label{eq:defRkstar}
{\mathcal R}_n^*:=\left\{x\in\mathcal{R}_n:\
\exists\lim_{r\to 0^+}\frac{\meas(B_r(x))}{\omega_nr^n}\in (0,\infty)\right\}.
\end{equation}
Then $\meas(\mathcal{R}_n\setminus\mathcal{R}_n^*)=0$, $\meas\res\mathcal{R}_n^*$ and 
$\mathcal{H}^n\res\mathcal{R}_n^*$ are mutually absolute continuous and
\begin{equation}\label{eq:gooddensity}
\lim_{r\to 0^+}\frac{\meas(B_r(x))}{\omega_nr^n}=\theta(x)
\qquad\text{for $\meas$-a.e. $x\in\mathcal{R}_n^*$,}
\end{equation}
\begin{equation}\label{eq:goodlimitRN}
\lim_{r\to 0^+} \frac{\omega_nr^n}{\meas(B_{r}(x))}=1_{\mathcal{R}^*_{n}}(x)
\frac{1}{\theta(x)}
\qquad\text{for $\meas$-a.e. $x\in X$.}
\end{equation}
Moreover  $\mathcal{H}^n(\mathcal{R}_n\setminus\mathcal{R}_n^*)=0$ 
if $n=N$.
\end{theorem}


\subsection{Second order differential structure and Riemannian metric}
Let $(X, \dist, \meas)$ be an $\RCD(K, \infty)$ space.

Inspired by \cite{Weaver}, the theory of the second order differential structure on $(X, \dist, \meas)$ based on $L^2$-normed modules is established in \cite{Gigli}.
To keep short presentations in the paper, we omit several notions, for instance, the spaces of $L^2$-vector fields denoted by $L^2(T(X, \dist, \meas))$ and of $L^2$-tensor fields of type $(0, 2)$ denoted by $L^2((T^*)^{\otimes 2}(X, \dist, \meas))$.
See \cite{Gigli} for the detail. We denote the pointwise Hilbert-Schmit norm and the pointwise scaler product by $|T|_{HS}$ and $\langle T, S \rangle$ respectively (see also \cite[Subsect.3.2]{Gigli} and \cite[Sect.$10$]{AmbrosioHonda}).

One of the important results in \cite{Gigli} we will use later is that for all $f \in D(\Delta)$, the Hessian $\mathrm{Hess}_f \in L^2((T^*)^{\otimes 2}(X, \dist, \meas))$ is well-defined and satisfies
\begin{align}
\langle \mathrm{Hess}_f, \dist f_1\otimes \dist f_2\rangle &=\frac{1}{2}\left(\langle \nabla f_1, \nabla \langle \nabla f, \nabla f_2\rangle\rangle +\langle \nabla f_2, \nabla \langle \nabla f, \nabla f_1\rangle\rangle -\langle \nabla f, \nabla \langle \nabla f_1, \nabla f_2\rangle\rangle \right) \nonumber \\
&\qquad \text{for $\meas-a.e. \, x \in X$}, \quad \forall f_i \in \mathrm{Test}F(X, \dist, \meas)
\end{align}
and the Bochner inequality with the Hessian term
\begin{equation}
\frac{1}{2}\Delta |\nabla f|^2\ge |\mathrm{Hess}_f|_{HS}^2+\langle \nabla \Delta f, \nabla f\rangle +K|\nabla f|^2
\end{equation}
in the weak sense (see \cite[Sect.4]{Gigli}).
In particular
\begin{equation}\label{eq:pp}
\int_X|\mathrm{Hess}_f|_{HS}^2\di \meas \le \int_X\left( (\Delta f)^2 -K |\nabla f|^2\right)\di \meas.
\end{equation}

Let us introduce the notion of Riemannian metrics on $(X, \dist, \meas)$. In order to simplify our argument we assume that $(X, \dist, \meas)$ is an $\RCD(K, N)$ space with $n=\dim_{\dist, \meas}(X)$ and $\mathrm{diam} (X, \dist)<\infty$ below.
Although we defined the notion as a bilinear form on $L^2(T(X, \dist, \meas))$ in \cite{AHPT}, we adopt an equivalent formulation by using tensor fields in this paper. Moreover we consider only $L^2$-ones, which is enough for our purposes.
\begin{definition}[$L^2$-Riemannian metric]
We say that $T \in L^2((T^*)^{\otimes 2}(X, \dist, \meas))$ is a \textit{Riemannian metric} if for all $\eta_i \in L^{\infty}(T^*(X, \dist, \meas))$ (which means that $\eta_i \in L^2(T^*(X, \dist, \meas))$ with $|\eta_i| \in L^{\infty}(X, \meas)$), it holds that 
\begin{equation}
\langle T, \eta_1\otimes \eta_2\rangle = \langle T, \eta_2\otimes \eta_1\rangle, \quad \langle T, \eta_1\otimes \eta_1\rangle \ge 0 \qquad \text{for $\meas$-a.e. $x\in X$}
\end{equation}
and that if $\langle T, \eta_1\otimes \eta_1\rangle=0$ for $\meas$-a.e. $x \in X$, then $\eta_1 =0$ in $L^2(T^*(X, \dist, \meas))$.  
\end{definition}
\begin{proposition}[The canonical metric $g$]\label{prop:gcanonical}
There exists a unique Riemannian metric $g \in L^2((T^*)^{\otimes 2}(X, \dist, \meas))$ such that
$$
\langle g, \dist f_1 \otimes \dist f_2\rangle=\langle\nabla f_1,\nabla f_2\rangle\qquad\text{for $\meas$-a.e. on $X$}
$$
for all Lipschitz functions $f_i$ on $X$. Then it holds that
\begin{equation}
|g|_{HS}=\sqrt{n} \qquad \text{for $\meas$-a.e. $x \in X$.}
\end{equation}
\end{proposition}
Note that for $T \in L^2((T^*)^{\otimes 2}(X, \dist, \meas))$, \textit{the trace} $\mathrm{tr}(T) \in L^2(X, \meas)$ is $\mathrm{tr}(T):=\langle T, g\rangle$.
The following result proved in \cite[Prop.3.2]{Han} will play a crucial role in the proof of Theorem \ref{thm:equiv}.
\begin{theorem}[Laplacian is trace of Hessian under maximal dimension]\label{thm:hanresult}
Assume that $N$ is an integer with $\mathrm{dim}_{\dist, \meas}(X)=N$. Then
for all $f \in D(\Delta)$ we see that 
\begin{equation}
\Delta f=\mathrm{tr}(\mathrm{Hess}_f) \quad \text{for $\meas$-a.e. $x \in X$}.
\end{equation}
\end{theorem}

Let us introduce the pull-back Riemannian metrics by embeddings via the heat kernel (see \cite[Prop.4.7]{AHPT}).
\begin{theorem}[The pull-back metrics]\label{prop:riemanexist}
For all $t>0$ there exists a unique Riemannian metric $g_t \in L^2((T^*)^{\otimes 2}(X, \dist, \meas))$ such that 
\begin{align}\label{eq:rrcd pull back}
\int_X\langle g_t, \eta_1 \otimes \eta_2\rangle \di\meas & = \int_X
\int_X\langle \dist_xp(x, y, t), \eta_1(x)\rangle \langle \dist_xp(x, y, t), \eta_2(x)\rangle\di\meas(x)\di\meas (y), \nonumber\\
&\quad \quad \quad \quad \quad  \forall \eta_i \in L^{\infty}(T^*(X, \dist, \meas)).
\end{align}
Moreover it is representable as the $HS$-convergent series
\begin{equation}\label{eq:explicit expression}
g_t=\sum_{i=1}^{\infty}e^{-2\lambda_it}\dist \phi_i \otimes \dist \phi_i
\qquad\text{in $L^2((T^*)^{\otimes 2}(X, \dist, \meas))$.} 
\end{equation}
Finally the rescaled metric
$t\meas (B_{\sqrt{t}}(\cdot))g_t$ satisfies
\begin{equation}\label{eq:riem est}
t\meas (B_{\sqrt{t}}(\cdot)) g_t\le C(K,N)g\qquad\forall t\in (0, 1),
\end{equation}
which means that for all $\eta \in L^{\infty}(T^*(X, \dist, \meas))$,
$$
t\meas (B_{\sqrt{t}}(x)) \langle g_t, \eta\otimes \eta \rangle (x) \le C(K, N)|\eta|_{HS}^2(x), \qquad \text{for $\meas$-a.e. $x \in X$}.
$$
\end{theorem} 
Note that since 
$$
\mathrm{Test}(T^*)^{\otimes 2}(X, \dist, \meas):=\left\{\sum_{i \ge 1}^kf_{1, i}\dist f_{2, i} \otimes \dist f_{3, i}; k \in \mathbb{N}, f_{j, i} \in \mathrm{Test}F(X, \dist, \meas)\right\}
$$
is dense in $L^2((T^*)^{\otimes 2}(X, \dist, \meas))$ (\cite[(3.2.7)]{Gigli}), it is easily checked that 
\begin{equation}
\int_X\langle g_t, T\rangle \di\meas =\int_X\int_X \langle \dist_xp \otimes \dist_xp, T\rangle \di \meas(x)\di \meas (y), \qquad \forall T \in L^2((T^*)^{\otimes 2}(X, \dist, \meas)).
\end{equation}
A main convergence result proved in \cite{AHPT} is the following;
\begin{theorem}[$L^p$-convergence to the original metric]\label{thm:conver}
We have 
\begin{equation}
\left|t\meas (B_{\sqrt{t}}( \cdot ))g_t - c_ng\right|_{HS} \to 0 \quad in\,\,L^p(X,  \meas)
\end{equation}
for all $p \in [1, \infty)$, where we recall (\ref{eq:const}) for the definition of $c_n$.
\end{theorem}
See \cite[Th.5.10]{AHPT} for their proofs of the results above.
It is worth pointing out that in general we can not improve this $L^p$-convergence to the $L^{\infty}$-one (see \cite[Rem.5.11]{AHPT}). 

We end this subsection by giving the following technical lemma.
\begin{lemma}\label{lem:co}
Let $(X, \dist, \meas)$ be a compact $\RCD(K, N)$ space with $n:=\mathrm{dim}_{\dist, \meas}(X)$.
Assume that there exists $C>0$ such that 
\begin{equation}\label{eq:domi}
\meas (B_r(x)) \ge Cr^n \quad \forall x \in X,\,\,\forall r \in (0, 1).
\end{equation}
Then as $t \to 0^+$ we see that 
\begin{equation}\label{eq:0}
t^{(n+2)/2}p(x, x, t) \to 0 \quad in\,\, H^{1, 2}(X, \dist, \meas).
\end{equation}
and that 
\begin{equation}\label{eq:00}
\left|\omega_nt^{(n+2)/2}g_t- c_n\frac{\di \mathcal{H}^n}{\di \meas}g\right|_{HS} \to 0 \qquad \text{in $L^p(X, \meas)$},\quad \forall p \in [1, \infty).
\end{equation}
\end{lemma}
\begin{proof}
By (\ref{eq:gaussian}) we see that for all $x \in X$ and all $t \in (0, 1)$;
\begin{align}\label{eq:000}
t^{(n+2)/2}p(x, x, t) \le \frac{t}{C}\meas (B_{\sqrt{t}}(x))p(x, x, t) \le \frac{C(K, N)}{C}t.
\end{align}
In particular $t^{(n+2)/2}p(x, x, t) \to 0$ in $C(X)$.
On the other hand (\ref{eq:100}) and (\ref{eq:explicit expression}) yield
\begin{align}\label{eq:01}
\int_X|\Delta_xp(x, x, t)|\di \meas(x) \le \sum_i4\lambda_ie^{-\lambda_it} =4\int_X\langle g, g_{t/2}\rangle \di \meas \le 4 \sqrt{n} \int_X|g_{t/2}|_{HS}\di \meas.  
\end{align}
In particular (\ref{eq:000}) and (\ref{eq:01}) yield
\begin{align}
&\int_X|\nabla_x (t^{(n+2)/2}p(x, x, t))|^2\di \meas (x) \nonumber \\
&=-\int_Xt^{n+2}p(x, x, t)\Delta_xp(x, x, t)\di \meas(x) \nonumber \\
&\le \frac{4C(K, N)\sqrt{n}}{C}t \int_X|t^{(n+2)/2}g_{t/2}|_{HS}\di \meas(x) \nonumber \\
&\le  \frac{4C(K, N)\sqrt{n}}{C^2}2^{(n+2)/2}t \int_X|(t/2)\meas (B_{\sqrt{t/2}}(x))g_{t/2}|_{HS}\di \meas(x) \to 0
\end{align}
as $t \to 0^+$, where we used (\ref{eq:riem est}). Thus we get (\ref{eq:0}).

Next let us prove (\ref{eq:00}).
First let us remark that (\ref{eq:domi}) yields $\mathcal{H}^n \ll \meas$. Combining this with Theorem \ref{thm:RN} shows that $\frac{\di \mathcal{H}^n}{\di \meas} \in L^{\infty}(X, \meas)$ and that as $r \to 0^+$,
\begin{equation}
\frac{\omega_nr^n}{\meas (B_r(x))} \to \frac{\di \mathcal{H}^n}{\di \meas}(x) \qquad \text{for $\meas$-a.e. $x \in X$.}
\end{equation}
Then since as $t \to 0$
\begin{align}
&\int_X\left|\omega_nt^{(n+2)/2}g_t- t\meas (B_{\sqrt{t}}(x))\frac{\di \mathcal{H}^n}{\di \meas}(x)g_t\right|_{HS}^p\di \meas \nonumber \\
&= \int_X\left| \frac{\omega_n\sqrt{t}^n}{\meas (B_{\sqrt{t}}(x))} - \frac{\di \mathcal{H}^n}{\di \meas}(x)\right|^p |t\meas (B_{\sqrt{t}}(x))g_t|_{HS}^p\di \meas \nonumber \\
&\le C(K, N)^p\int_X\left| \frac{\omega_n\sqrt{t}^n}{\meas (B_{\sqrt{t}}(x))} - \frac{\di \mathcal{H}^n}{\di \meas}(x)\right|^p\di \meas \to 0,
\end{align}
we conclude because of Theorem \ref{thm:conver}, 
where we used the dominated convergence theorem.
\end{proof}
\section{Laplacian on $(X, g_t, \meas)$}\label{sec3}
Let $(X, \dist, \meas)$ be a compact $\RCD(K, N)$ space.
We rewrite our new differential operators;
\begin{definition}[Laplacian on $(X, g_t, \meas)$]
For all $f \in D(\Delta)$ we define \textit{the Laplacian $\Delta^t f$ associated to} $g_t$ by
\begin{equation}\label{eq:twisted}
\Delta^tf:=\langle \mathrm{Hess}_f, g_t\rangle+\frac{1}{4}\langle \nabla_x\Delta_xp(x, x, 2t), \nabla f\rangle \in L^1(X, \meas).
\end{equation}
\end{definition}
Let us start calculation.
\begin{lemma}\label{lem:1}
For all $f \in D(\Delta)$ and $\psi \in H^{1, 2}(X, \dist, \meas)$ we have
\begin{align}\label{eq:1}
&\int_X\int_X\psi (x)\langle \nabla_xp, \nabla_x\langle \nabla_xp, \nabla f\rangle \rangle \di\meas(x) \di \meas(y) \nonumber \\
&= -\int_X\langle g_t, \dist f\otimes \dist \psi\rangle \di \meas +\frac{1}{4}\int_X \mathrm{div} (\psi \nabla f) \frac{\di}{\di t}p(x, x, 2t)\di \meas.
\end{align}
\end{lemma}
\begin{proof}
Note that
\begin{align}\label{al:1}
&\int_X\int_X\psi(x)\langle \nabla_xp, \nabla_x\langle \nabla_xp, \nabla f\rangle \rangle \di\meas(x) \di \meas(y) \nonumber \\
&=-\int_X\int_X\mathrm{div}_x (\psi \nabla_x p)\langle \nabla_xp, \nabla f\rangle \di \meas(x) \di \meas (y) \nonumber \\
&=-\int_X\int_X(\langle \nabla \psi, \nabla_xp\rangle +\psi(x)\Delta_xp)\langle \nabla_xp, \nabla f\rangle \di \meas(x) \di \meas(y) \nonumber \\
&=-\int_X\langle g_t, \dist f \otimes \dist \psi \rangle \di \meas -\int_X\int_X\psi (x)\Delta_xp \langle \nabla_xp, \nabla f\rangle \di \meas(x) \di \meas (y)
\end{align}
and that
\begin{align}\label{eq:h2}
&-\int_X\int_X\psi (x)\Delta_xp \langle \nabla_xp, \nabla f\rangle \di \meas(x) \di \meas (y) \nonumber \\
&=-\int_X\int_X\psi (x)\left( -\sum_{i \ge 0}\lambda_ie^{-\lambda_it}\phi_i(x)\phi_i(y)\right) \left(\sum_{i \ge 0}e^{-\lambda_it}\phi_i(y)\langle \nabla \phi_i, \nabla f\rangle (x) \right)\di \meas (x) \di \meas(y)\nonumber \\
&=\sum_{i \ge 0}\lambda_ie^{-2\lambda_it}\int_X\psi (x)\phi_i(x)\langle \nabla f, \nabla \phi_i\rangle(x) \di \meas(x) \nonumber \\
&=\frac{1}{2}\sum_{i \ge 0}\lambda_ie^{-2\lambda_it}\int_X\langle \psi \nabla f, \nabla \phi_i^2\rangle \di \meas =-\frac{1}{2}\int_X\mathrm{div} (\psi \nabla f) \left( \sum_{i \ge 0}\lambda_ie^{-2\lambda_it}\phi_i^2\right)\di \meas
\end{align}
On the other hand since 
\begin{equation}
\frac{\di}{\di t}p(x, x, 2t) = -2\sum_{i \ge 0}\lambda_ie^{-2\lambda_it}\phi_i(x)^2,
\end{equation}
we have 
\begin{equation}
-\frac{1}{2}\int_X\mathrm{div} (\psi \nabla f) \left( \sum_i\lambda_ie^{-2\lambda_it}\phi_i^2\right)\di \meas = \frac{1}{4}\int_X\mathrm{div}(\psi \nabla f) \frac{\di}{\di t}p(x, x, 2t) \di \meas,
\end{equation}
which completes the proof because of (\ref{al:1}) and (\ref{eq:h2}).
\end{proof}
\begin{lemma}\label{lem:2}
For all $f \in D(\Delta)$ and $\psi \in H^{1, 2}(X, \dist, \meas)$ we have
\begin{align}\label{eq:1}
&-\frac{1}{2}\int_X\int_X\psi (x)\langle \nabla f, \nabla_x|\nabla_xp|^2\rangle\di \meas(x) \di \meas(y) \nonumber\\
&= -\frac{1}{4}\int_X \mathrm{div} (\psi \nabla f) \frac{\di}{\di t}p(x, x, 2t)\di \meas(x) +\frac{1}{4}\int_X \mathrm{div} (\psi \nabla f) \Delta_xp(x, x, 2t)\di \meas(x).
\end{align}
\end{lemma}
\begin{proof}
By Lemma \ref{lem:density} it is enough to prove (\ref{eq:1}) under assuming $f \in \mathrm{Test}(X, \dist, \meas)$.

First assume $\psi \in \mathrm{Test}F(X, \dist, \meas)$. Let $\phi:=\mathrm{div} (\psi \nabla f) \in H^{1, 2}(X, \dist, \meas)$.
Then 
\begin{align}
&-\frac{1}{2}\int_X\int_X\psi (x)\langle \nabla_x f, \nabla_x|\nabla_xp|^2\rangle \di \meas(x) \di \meas(y) \nonumber\\
&=\frac{1}{2}\int_X\int_X\phi (x) |\nabla_xp|^2\di \meas(x) \di \meas(y) \nonumber\\
&=\frac{1}{2}\int_X\int_X \langle \nabla_xp, \nabla_x (\phi p)\rangle \di \meas(x) \di \meas(y)-\frac{1}{2}\int_X\int_X \langle p\nabla_xp, \nabla_x \phi \rangle \di \meas(x) \di \meas(y) \nonumber\\
&=-\frac{1}{2}\int_X\int_X\phi (x) p\Delta_xp\di \meas (x) \di \meas (y)-\frac{1}{4}\int_X\int_X\langle \nabla_x p^2, \nabla_x \phi \rangle \di \meas (x) \di \meas (y) \nonumber \\
&=-\frac{1}{2}\int_X\int_X\phi (x) p\frac{\di}{\di t}p\di \meas (x) \di \meas (y)+\frac{1}{4}\int_X\int_X \phi\Delta_x p^2 \di \meas (x) \di \meas (y) \nonumber \\
&=-\frac{1}{4}\int_X\phi(x)\left(\int_X \frac{\di}{\di t}p^2 \di \meas (y)\right)\di \meas (x)+\frac{1}{4}\int_X \phi (x) \left( \int_X \Delta_x p^2 \di \meas (y)\right) \di \meas (x) \nonumber \\
&=-\frac{1}{4}\int_X\phi(x)\frac{\di}{\di t}\left(\int_X p^2 \di \meas (y)\right)\di \meas (x)+\frac{1}{4}\int_X \phi (x) \Delta_x\left( \int_X p^2 \di \meas (y)\right) \di \meas (x) \nonumber \\
&=-\frac{1}{4}\int_X\phi(x)\frac{\di}{\di t}p(x, x, 2t)\di \meas(x)+\frac{1}{4}\int_X \phi (x) \Delta_xp(x, x, 2t) \di \meas (x),
\end{align}
which proves (\ref{eq:1}), where we used (\ref{trans}).

Finally let us prove (\ref{eq:1}) for general $\psi \in H^{1, 2}(X, \dist, \meas)$.
Let $\psi_L:= (-L)  \vee \psi \wedge L \in H^{1, 2}(X, \dist, \meas) \cap L^{\infty}(X, \meas)$.
Since (\ref{eq:1}) holds as $\psi=h_t (\psi_L)$ for all $t>0$, letting $L \to \infty$ and then letting $t \to 0^+$ shows the desired claim.
\end{proof}
\begin{theorem}[Integration by parts on $(X, g_t, \meas)$]\label{thm:green}
For all $f \in D(\Delta)$ and $\psi \in H^{1, 2}(X, \dist, \meas) \cap L^{\infty}(X, \meas)$ we have
\begin{equation}\label{eq:div}
\int_X\langle g_t, \dist \psi \otimes \dist f\rangle \di \meas = -\int_X\psi\Delta^tf \di \meas.
\end{equation}
\end{theorem}
\begin{proof}
Lemmas \ref{lem:1} and \ref{lem:2} yield
\begin{align*}
&\int_X\psi \langle \mathrm{Hess}_f, g_t\rangle \di \meas \nonumber \\
&=\int_X\int_X\psi (x)\langle \mathrm{Hess}_f, \dist_xp \otimes \dist_xp\rangle \di \meas(x) \di \meas(y) \nonumber \\
&=\int_X\int_X\psi (x)\langle \nabla_xp, \nabla_x\langle \nabla_xp, \nabla f\rangle \rangle \di \meas(x) \di \meas(y) -\frac{1}{2}\int_X\int_X\psi (x)\langle \nabla f, \nabla_x|\nabla_xp|^2\rangle\di \meas(x) \di \meas(y) \nonumber \\
&=-\int_X\langle g_t, \dist f \otimes \dist \psi \rangle \di \meas +\frac{1}{4}\int_X\mathrm{div} (\psi \nabla f)\Delta_xp(x, x, 2t)\di \meas (x) \nonumber \\
&=-\int_X\langle g_t, \dist f \otimes \dist \psi \rangle \di \meas-\frac{1}{4}\int_X\psi \langle \nabla f, \nabla_x \Delta_xp(x, x, 2t)\rangle \di \meas (x),
\end{align*}
which proves (\ref{eq:div}).
\end{proof}
\section{Characterization of noncollapsed $\RCD$ spaces}\label{sec4}
\subsection{Proof of Theorem \ref{thm:equiv}}

Assume that (2) holds. Then the Bishop-Gromov inequality yields that (\ref{eq:meas}) holds.
Moreover it follows from Theorem \ref{thm:hanresult} that (\ref{eq:lap}) holds. This proves the implication from (2) to (1).

Next we assume that (1) holds. Fix a nonconstant eigenfunction $f$ of $-\Delta$ on $(X, \dist, \meas)$ with the eigenvalue $\lambda > 0$.
Applying Theorem \ref{thm:green} as $\psi \equiv 1$ shows 
\begin{equation}\label{eq:12}
0=-\int_X\langle \mathrm{Hess}_f, \omega_nt^{(n+2)/2}g_t\rangle \di \meas -\frac{1}{4}\int_X \langle \nabla \omega_nt^{(n+2)/2}p(x, x, 2t), \nabla \Delta f(x)\rangle \di \meas (x).
\end{equation}
Lemma \ref{lem:co} yields that as $t \to 0^+$, the first term of the RHS of (\ref{eq:12}) converges to 
\begin{equation}\label{eq:13}
-c_n\int_X\mathrm{tr}(\mathrm{Hess}_f)\frac{\di \mathcal{H}^n}{\di \meas} \di \meas =c_n\lambda \int_Xf\frac{\di \mathcal{H}^n}{\di \meas} \di \meas.
\end{equation}

On the other hand Lemma \ref{lem:co} yields that as $t \to 0^+$, the second term of the RHS of (\ref{eq:12}) converges to $0$. Thus (\ref{eq:13}) is equal to $0$, in particular $\frac{\di \mathcal{H}^n}{\di \meas}$ is $L^2$-orthogonal to $f$, which shows that $\frac{\di \mathcal{H}^n}{\di \meas}$ must be a constant function.  

For all $\psi \in D(\Delta)$ since
\begin{equation}\label{eqqq}
\psi =\sum_{i \ge 0} \left(\int_X\psi \phi_i\di \meas\right) \phi_i \qquad \text{in $H^{1, 2}(X, \dist, \meas)$}
\end{equation}
and 
\begin{equation}\label{eqqqq}
\Delta \psi =-\sum_{i \ge 0} \lambda_i \left(\int_X\psi \phi_i\di \meas\right) \phi_i \qquad \text{in $L^{2}(X, \meas)$}
\end{equation}
(c.f. appendices of \cite{AHPT} and \cite{Honda0}),  combining (\ref{eqqq}) and (\ref{eqqqq}) with (\ref{eq:pp}) yields 
\begin{equation}
\mathrm{Hess}_{\psi}=\sum_{i \ge 0} \left(\int_X\psi \phi_i\di \meas\right) \mathrm{Hess}_{\phi_i} \qquad \text{in $L^2((T^*)^{\otimes 2}(X, \dist, \meas))$.}
\end{equation}
In particular 
\begin{align}
\mathrm{tr}(\mathrm{Hess}_{\psi}) &= \left\langle \sum_{i \ge 0} \left(\int_X\psi \phi_i\di \meas\right) \mathrm{Hess}_{\phi_i}, g\right\rangle \nonumber \\
&=\sum_{i \ge 0} \left(\int_X\psi \phi_i\di \meas\right)\langle \mathrm{Hess}_{\phi_i}, g\rangle \nonumber \\
&=\sum_{i \ge 0} \left(\int_X\psi \phi_i\di \meas \right)\Delta \phi_i =-\sum_{i \ge 0} \lambda_i \left(\int_X\psi \phi_i\di \meas\right) \phi_i=\Delta \psi \quad \text{in $L^2(X, \meas)$.}
\end{align}
Therefore if $\Delta \psi \in H^{1, 2}(X, \dist, \meas)$, then in the weak sense it holds that
\begin{align}
\frac{1}{2}\Delta |\nabla \psi|^2 &\ge |\mathrm{Hess}_{\psi}|^2 +\langle \nabla \Delta \psi, \nabla \psi\rangle +K|\nabla \psi|^2 \nonumber \\
&\ge \frac{(\mathrm{tr}(\mathrm{Hess}_{\psi}))^2}{n}+\langle \nabla \Delta \psi, \nabla \psi\rangle +K|\nabla \psi|^2 \nonumber \\
&=\frac{(\Delta \psi)^2}{n}+\langle \nabla \Delta \psi, \nabla \psi\rangle +K|\nabla \psi|^2.
\end{align} 
This shows that $(X, \dist, \meas)$ is an $\RCD(K, n)$ space. Thus we get (2).
\,\,\, $\square$
\subsection{Witten Laplacian on $\RCD$ spaces}\label{4.2}
Let us recall that for a closed manifold $(M^n, g)$ with a smooth function $\phi \in C^{\infty}(M^n)$, the corresponding Laplacian of the weighted space $(M^n, g, e^{-\phi}\mathrm{vol}_g)$ is the \textit{Witten Laplacian} $\Delta_{\phi}$, that is, 
\begin{equation}
\int_{M^n}\langle \nabla f_1, \nabla f_2\rangle e^{-\phi}\di \mathrm{vol}_g=-\int_{M^n}f_1\Delta_{\phi}f_2 e^{-\phi}\di \mathrm{vol}_g \qquad \forall f_i \in C^{\infty}(M^n),
\end{equation}
where $\Delta_{\phi}f:=\mathrm{tr}(\mathrm{Hess}_f)-\langle \nabla \phi, \nabla f\rangle$.
By using the formula (\ref{eq:div}) we can prove an analogous result in the nonsmooth setting. 
Compare with \cite[Prop.3.5]{Han2}. 
\begin{theorem}[Witten Laplacian on $\RCD$ spaces]\label{thm:be}
Let $n \in [1, \infty)$ and let $(X, \dist)$ be a compact metric space satisfying that there exists $C>0$ such that 
\begin{equation}\label{eq:11}
\mathcal{H}^n(B_r(x)) \ge Cr^n \quad \forall x \in X,\,\,\forall r \in (0, 1).
\end{equation}
If $(X, \dist, e^{-\phi} \mathcal{H}^n)$ is an $\RCD(K, N)$ space for some $N \in [1, \infty)$ and some $\phi \in \mathrm{LIP}(X, \dist)$, then for all $f \in D(\Delta )$ we have
\begin{equation}\label{eq:be1}
\Delta f= \mathrm{tr}(\mathrm{Hess}_f) -\langle \nabla \phi, \nabla f\rangle \qquad \text{in $L^2(X, \meas)$.}
\end{equation}
\end{theorem}
\begin{proof}
Let $\meas:=e^{-\phi}\mathcal{H}^n$.
By Lemma \ref{lem:density} it is enough to prove (\ref{eq:be1}) under assuming $f \in \mathrm{Test}F(X, \dist, \meas)$. 
Note that 
\begin{equation}
\meas (B_r(x))=\int_{B_r(x)}e^{-\phi}\di \mathcal{H}^n \ge Ce^{-\max \phi}r^n, \quad \forall x \in X,\,\,\forall r \in (0, 1).
\end{equation}
Then by an argument similar to the proof of Theorem \ref{thm:equiv} we see that for all $\psi \in \mathrm{Test}F(X, \dist, \meas)$
\begin{equation}\label{eq:14}
\int_X\langle \nabla \psi, \nabla f\rangle e^{\phi}\di \meas = -\int_X\psi \mathrm{tr}(\mathrm{Hess}_f) e^{\phi}\di \meas.
\end{equation}
Since the LHS of (\ref{eq:14}) is equal to 
\begin{align}
\int_X\langle \nabla (\psi e^{\phi}), \nabla f\rangle \di \meas -\int_X\langle \nabla \phi, \nabla f\rangle \psi e^{\phi}\di \meas =-\int_X \psi e^{\phi}\Delta f\di \meas -\int_X\langle \nabla \phi, \nabla f\rangle \psi e^{\phi}\di \meas,
\end{align}
we have 
\begin{equation}
\int_X\left( -\Delta f - \langle \nabla \phi, \nabla f\rangle + \mathrm{tr}(\mathrm{Hess}_f)\right) \psi e^{\phi}\di \meas =0
\end{equation}
which completes the proof of (\ref{eq:be1}) because $\psi$ is arbitrary.
\end{proof}
We end this paper by giving another characterization of non-collapsed $\RCD$ spaces;
\begin{corollary}
Let $n, N \in [1, \infty)$ and let $(X, \dist, \mathcal{H}^n)$ be a compact $\RCD(K, N)$ space.
Then the following two conditions are equivalent;
\begin{enumerate}
\item There exists $C>0$ such that 
\begin{equation}
\mathcal{H}^n(B_r(x)) \ge Cr^n \quad \forall x \in X,\,\, \forall r \in (0, 1).
\end{equation} 
\item $(X, \dist, \mathcal{H}^n)$ is an $\RCD(K, n)$ space, that is, it is a non-collapsed space.
\end{enumerate}
\end{corollary}
\begin{proof}
The implication from (2) to (1) is trivial because of the Bishop-Gromov inequality.

Assume that (1) holds.
Then applying Theorem \ref{thm:be} as $\phi \equiv 0$ yields that (\ref{eq:lap}) holds.
Therefore Theorem \ref{thm:equiv} shows that (2) holds.
\end{proof}

\end{document}